\documentclass[11pt]{amsart}

\usepackage{amsfonts, amssymb, amscd}
\numberwithin{equation}{section}

\usepackage[symbol]{footmisc}

\usepackage{bm}
\usepackage{verbatim}
\usepackage{amssymb}
\usepackage{mathrsfs}
\usepackage{graphicx}
\usepackage[nospace,noadjust]{cite}
\usepackage{tikz-cd}
\usepackage{subcaption}
\usepackage{subfiles}
\usepackage[toc,page]{appendix}
\usepackage{mathtools}
\usepackage{comment}
\usepackage{enumerate}
\usepackage{enumitem}

\usepackage{graphicx}
\graphicspath{{images/}}

\usepackage{appendix}
\usepackage{hyperref}

\usepackage{todonotes}

\newcommand{\pr}{^{\prime}}
\newcommand{\prpr}{^{\prime\prime}}

\newcommand{\Cc}{\mathbb{C}}
\newcommand{\KK}{\mathbb{K}}

\newcommand{\Qq}{\mathbb{Q}}
\newcommand{\QQ}{\mathbb{Q}}
\newcommand{\Rr}{\mathbb{R}}

\newcommand{\Zz}{\mathbb{Z}}

\newcommand{\Span}{\operatorname{Span}}
\newcommand{\alct}{a\text{-}\operatorname{lct}}

\newcommand{\Center}{\operatorname{center}}

\newcommand{\mld}{{\rm{mld}}}

\newcommand{\lct}{\operatorname{lct}}

\newcommand{\Supp}{\operatorname{Supp}}

\newcommand{\mult}{\operatorname{mult}}

\newcommand{\lf}{\lfloor}
\newcommand{\rf}{\rfloor}

\newcommand{\Ff}{\mathcal{F}}
\newcommand{\Oo}{\mathcal{O}}
\newcommand{\Ii}{{\Gamma}}

\newcommand{\xto}{\xrightarrow{f}}

 \usepackage{todonotes}

\newcommand\chen[1]{\todo[color=pink!40]{#1}} 

\newtheorem{thm}{Theorem}[section]
\newtheorem{conj}[thm]{Conjecture}

\newtheorem{lem}[thm]{Lemma}

\newtheorem{prop}[thm]{Proposition}

\theoremstyle{definition}
\newtheorem{defn}[thm]{Definition}
\newtheorem{rem}[thm]{Remark}

\theoremstyle{definition}

\begin{document}

\title{Boundedness of ($\epsilon, n$)-Complements for projective generalized pairs of Fano type}


\author{Guodu Chen}
\address{Guodu Chen, Beijing International Center for Mathematical Research, Peking University, Beijing 100871, China}
\email{gdchen@pku.edu.cn}

\author{Qingyuan Xue}
\address{Qingyuan Xue, Department of Mathematics, The University of Uath, Salt Lake City, UT 84112, USA}
\email{xue@math.utah.edu}

\begin{abstract}
We show the existence of $(\epsilon,n)$-complements for $(\epsilon,\Rr)$-complementary projective generalized pairs of Fano type $(X,B+M)$ when either the coefficients of $B$ and $\mu_j$ belong to a finite set or the coefficients of $B$ belong to a DCC set and $M'\equiv 0$, where $M'=\sum\mu_jM_j'$ and $M_j'$ are b-Cartier nef divisors. 
\end{abstract}

\date{\today}

\maketitle
\pagestyle{myheadings}\markboth{\hfill  Guodu Chen and Qingyuan Xue \hfill}{\hfill Boundedness of ($\epsilon, n$)-Complements for projective generalized pairs of Fano type\hfill}

\tableofcontents

\section{Introduction}
We work over the field of complex numbers $\Cc.$ 

In \cite{Sho92}, Shokurov introduced the theory of complements to investigate log flips of threefolds, and it turns out that the theory plays an important role in birational geometry. The theory of complements has been studied in recent years; for example, see \cite{Bir19,HLS19} for the boundedness of log canonical complements, and see also \cite{Sho00,PS01,Bir04,PS09} for more on boundedness of complements. We refer the readers to \cite{HLS19,CH20} for the applications of the theory of complements.

In recent years a new concept of space, generalized pairs, has evolved. Generalized pairs appear naturally in the study of birational geometry in higher dimension.
They were first introduced in \cite{BZ16}, and we refer the readers to \cite{Bir20} for more motivations and applications. It is natural to consider the boundedness of complements for generalized pairs. Indeed, Birkar proved the boundedness of generalized log canonical complements for generalized pairs when the coefficients of the boundaries belong to a hyperstandard set \cite{Bir19}, and the first author proved the boundedness of generalized log canonical complements for generalized pairs with DCC coefficients \cite{Chen20} by using the Diophantine approximation in \cite{HLS19}. It is worth mentioning that in \cite{Chen20}, the first author have studied the complements in a more general setting. More precisely, the nef part of the generalized pair is allowed to have irrational coefficients; see Definition \ref{defn:complements}.

However, the above results were mainly focus on the boundedness of generalized log canonical complements. So it is natural to consider the boundedness of complements with good singularities, that is, generalized $\epsilon$-log canonical complements (See Definition \ref{defn:complements}). In this paper, we deal with the following deep conjecture regarding the boundedness of $(\epsilon,n)$-complements for generalized pairs which is an analogue of \cite[Conjecture 1.1]{CH20} in the setting of generalized pairs.

\begin{conj}
	\label{conj a complement}
	Let $d,p$ be two positive integers, $\epsilon$ a non-negative real number, and $\Ii \subseteq [0,1]$ a DCC set. Then there exists a positive integer $n$ divisible by $p$ depending only on $d, p, \epsilon$ and $\Ii$ satisfying the following.
	
	Assume that $(X,B+M)$ is a generalized pair with data $X'\xto X\to Z$ and $M'=\sum \mu_j M_j'$, $X\to Z$ a contraction and $z\in Z$ a (not necessarily closed) point such that
	\begin{enumerate}     
		\item $\dim X=d$,
		\item $X$ of Fano type over $Z$,
		\item $B\in \Ii$, that is, the coefficients of $B$ belong to $\Ii$,
		\item $M_j'$ is a b-Cartier nef$/Z$ divisor and $\mu_j\in\Ii$ for any $j$, and
		\item $(X/Z\ni z,B+M)$ is $(\epsilon,\Rr)$-complementary.
	\end{enumerate} 
	Then there is an $(\epsilon,n)$-complement $(X/Z\ni z,B^{+}+M^+)$ of $(X/Z\ni z,B+M).$ Moreover, if $\Span_{\Qq_{\ge0}}(\bar{\Ii}\cup\{\epsilon\}\backslash\Qq)\cap (\Qq\backslash\{0\})=\emptyset$, then we may pick $B^+\ge B$ and $\mu_j^+\ge\mu_j$ for any $j$, where $M^{+\prime}=\sum \mu_j^+M_j'$.
\end{conj}

\begin{rem}
	In Conjecture \ref{conj a complement}, $M'$ is allowed to have irrational coefficients while in \cite{Bir19} $M'$ is a $\Qq$-Cartier divisor whose Cartier index is fixed. The ``Moreover'' part is about monotonicity property of complements which is useful in applications especially when $\bar{\Gamma}\subseteq\Qq$ and does not hold in general.
\end{rem}

\begin{rem}
	Without condition (2), Conjecture \ref{conj a complement} fails even when $\dim X=2$; see \cite[Example 1.7]{CH20}.
\end{rem}

When $\epsilon=0$, the conjecture is proved in \cite{Chen20}. When $\epsilon$ is positive, we have some partial results.

\begin{thm}\label{thm: conj for global FT}
	Conjecture \ref{conj a complement} holds in the following cases:
	\begin{enumerate}
		\item $\epsilon=0$;
		\item $\epsilon>0$, $\dim Z=0$ and $\Gamma$ is a finite set; and
		\item $\epsilon>0$, $\dim Z=0$ and $M'\equiv 0$.
	\end{enumerate}
\end{thm}

In order to show Theorem \ref{thm: conj for global FT}, we study a new class of complements, namely $(\bm{\epsilon},n,\Ii_0)$-decomposable $(\epsilon,\Rr)$-complements. Note that when $\epsilon=\epsilon_i=0$, $(\bm{\epsilon},n,\Ii_0)$-decomposable $(\epsilon,\Rr)$-complements are the same as \cite[Definition 1.2]{Chen20}.

\begin{defn}
	Let $n$ be a positive integer, $\epsilon,\epsilon_i$ non-negative real numbers, $\Ii_0\subseteq(0,1]$ a finite set and $(X,B+M)$ a generalized pair with data $X'\xto X\to Z$ and $M'$. We say that $(X/Z\ni z,B^++M^+)$ is an \emph{$(\bm{\epsilon},n,\Ii_0)$-decomposable $(\epsilon,\Rr)$-complement} of $(X/Z\ni z,B+M)$ if
	\begin{enumerate}
		\item $(X/Z\ni z,B^++M^+)$ is an $(\epsilon,\Rr)$-complement of $(X/Z\ni z,B+M)$,
		\item $K_X+B^++M^+=\sum a_i(K_X+B_i^++M^+_i)$ for some boundaries $B_i^+$ and nef parts $M_i^{+\prime}$, and $a_i\in\Ii_0$ with $\sum a_i=1$ and $\sum a_i\epsilon_i\ge\epsilon$, and
		\item $(X/Z\ni z,B_i^++M_i^+)$ is an $(\epsilon_i,n)$-complement of itself for any $i$.
	\end{enumerate}
\end{defn}

As an important step in the proof of Theorem \ref{thm: conj for global FT}, we show the existence of $(\bm{\epsilon},n,\Ii_0)$-decomposable $(\epsilon,\Rr)$-complements under the conditons of Theorem \ref{thm: conj for global FT}. We remark that when $\epsilon=0$, the theorem is proved in \cite{Chen20}; see \cite[Theorem 1.3]{Chen20}.

\begin{thm}\label{thm: (n,I)complforfiniterat}
	Let $d$ be a positive integer, $\epsilon$ a positive real number and $\Ii\subseteq [0,1]$ a DCC set. Then there exist a positive integer $n$, non-positive real numbers $\epsilon_i$ and a finite set $\Ii_0\subseteq(0,1]$ depending only on $d,\epsilon$ and $\Ii$ satisfying the following. 
	
	Assume that $(X,B+M)$ is a generalized polarized pair with data $X'\xto X$ and $M'=\sum \mu_jM_j'$ such that
	\begin{enumerate}
		\item $\dim X=d$,	
		\item $X$ is of Fano type,
		\item $B\in\Ii$,
		\item $M_j'$ is a b-Cartier nef divisor and $\mu_j\in\Ii$ for any $j$,
		\item $(X,B+M)$ is $(\epsilon,\Rr)$-complementary, and
		\item either $\Gamma$ is a finite set or $M'\equiv 0$.
	\end{enumerate}
	Then there is an $(\bm{\epsilon},n,\Ii_0)$-decomposable $(\epsilon,\Rr)$-complement $(X,B^++M^+)$ of $(X,B+M)$. Moreover, if $\bar{\Ii}\subseteq\Qq$, then we may pick $\Ii_0=\{1\}$, and $(X,B^++M^+)$ is a monotonic $(\epsilon,n)$-complement of $(X,B+M)$.
\end{thm}

Since we work on $(\epsilon,n)$-complements, we should be careful with the singularities when we prove Theorem \ref{thm: conj for global FT}. A key observation is the uniform linearity of minimal log discrepancies for $(\epsilon,\Rr)$-complementary projective generalized pairs of Fano type when $\epsilon$ is positive. More precisely, we have the following result.

\begin{thm}[=Theorem \ref{thm:mldlinearity}(3)]\label{thm:linearity}
	Let $\epsilon$ be a positive real number, $d,c,m,l$ positive integers, $\bm{r}_0=(r_1,\dots,r_c)\in\Rr^c$ a point such that $r_0=1,r_1,\dots,r_c$ are linearly independent over $\Qq$, and $s_1,\ldots, s_{m+l}: \Rr^{c}\to\Rr$ $\Qq$-linear functions. Then there exist a positive real number $\delta$ and a $\Qq$-linear function $f(\bm{r}):\Rr^c\to\Rr$ depending only on $\epsilon,d,c,m,l,\bm{r}_0$ and $s_i$ satisfying the following.
	
	Assume that $(X,B(\bm{r}_0)+M(\bm{r}_0))$ is a projective generalized pair with data $X'\xto X$ and $M'(\bm{r}_0)=\sum_{j=1}^{l}s_{m+j}(\bm{r}_0)M_j'$ such that
	\begin{itemize}
		\item $\dim X$=d,
		\item $X$ is of Fano type,
		\item $B(\bm{r})=\sum_{i=1}^ms_i(\bm{r})B_i,$ and $B_i\ge0$ are Weil divisors,
		\item $M_j'$ are b-Cartier nef divisors, and
		\item $(X,B(\bm{r}_0)+M(\bm{r}_0))$ is $(\epsilon,\Rr)$-complementary.
	\end{itemize}
	Then there exists a prime divisor $E$ over $X$ such that
	$$\mld(X,B(\bm{r})+M(\bm{r}))=a(E,X,B(\bm{r})+M(\bm{r}))\ge f(\bm{r})$$
	for any $\bm{r}\in\Rr^c$ satisfying $||\bm{r}-\bm{r}_0||_\infty\le\delta$. 
\end{thm}

\medskip

\textit{Structure of the paper.}
We outline the organization of the paper. In Section \ref{section2}, we introduce some notations and tools which will be used in this paper, and prove certain basic results. In section \ref{section3}, we prove Theorem \ref{thm:linearity}. In section \ref{section4}, we prove Theorem \ref{thm: (n,I)complforfiniterat}. In section \ref{section5}, we prove Theorem \ref{thm: conj for global FT}.

\medskip

\textbf{Acknowledgements.} The first author would like to thank his advisor Chenyang Xu for constant support and encouragement. The second author would like to thank his advisor Christopher Hacon for his support. The authors would also like to thank Jingjun Han and Jihao Liu for useful discussions and comments.

\section{Preliminaries}\label{section2}
In this section, we will collect some definitions and preliminary results which will be used in this paper.
\subsection{Arithmetic of sets}
\begin{defn}[DCC and ACC sets]\label{def: DCC and ACC}        
	We say that $\Ii\subseteq\Rr$ satisfies the \emph{descending chain condition} $($DCC$)$ if any decreasing sequence $a_{1} \ge a_{2} \ge \cdots$ in $\Ii$ stabilizes. We say that $\Ii$ satisfies the \emph{ascending chain condition} $($ACC$)$ if any increasing sequence in $\Ii$ stabilizes.
\end{defn}

\subsection{Divisors}
Let $\KK$ be either the rational number field $\Qq$ or the real number field $\Rr$. On a normal variety $X$, a $\KK$-divisor is a finite formal $\KK$-linear combination $D=\sum d_{i} D_{i}$ of prime Weil divisors $D_{i}$, $\mult_{D_{i}} D$ denotes the coefficients $d_{i}$, and a $\KK$-Cartier divisor is a $\KK$-linear combination of Cartier divisors. An $\Rr$-divisor is called a \emph{boundary} (respectively \emph{sub-boundary}) divisor if its coefficients belong to $[0,1]$ (respectively $(-\infty, 1]$). We define $\lfloor D\rfloor=\sum \lfloor d_i\rfloor D_i, \{B\}=\sum \{d_i\}D_i$, and $||D||=\max_i\{|d_i|\}$. For any point $\bm{v}=(v_1,\dots,v_m)\in \Rr^m,$ we define $||\bm{v}||_{\infty}=\max_{1\le i\le m}\{|v_i|\}$.

Let $M$ be an $\Rr$-divisor on $X$. We say that $M$ is \emph{b-Cartier} if it is $\Qq$-Cartier and $\phi^*M$ is Cartier for some birational morphism $\phi:Y\to X$.

We use $\sim_{\KK}$ to denote the $\KK$-linear equivalence of two divisors. For a projective morphism $X\to Z$, we use $\sim_{\KK,Z}$ to denote the relative $\KK$-linear equivalence. 

\begin{defn}[b-divisors]
	Let $X$ be a variety. A \emph{b-$\Rr$-Cartier b-divisor} over $X$ is the choice of a projective birational morphism $Y\to X$ from a normal variety $Y$ and an $\Rr$-Cartier divisor $M$ on $Y$ up to the following equivalence: another projective birational morphism $Y '\to  X$ from a normal variety $Y'$ and an $\Rr$-Cartier divisor $M'$ define the same b-$\Rr$-Cartier b-divisor if there is a common resolution $W\to Y$ and $W\to Y'$ on which the pullbacks of $M$ and $M'$ coincide.
	
	A b-$\Rr$-Cartier b-divisor represented by some $Y\to X$ and $M$ is \emph{b-Cartier} if $M$ is b-Cartier, i.e., its pullback to some resolution is Cartier.
\end{defn}

\subsection{Generalized pairs}

\begin{defn}   
	We say $\pi: X \to Z$ is a \emph{contraction} if $\pi$ is a projective morphism, and $\pi_*\Oo_X = \Oo_Z$ $(\pi$ is not necessarily birational$)$. In particular, $\pi$ is surjective and has connected fibers.
\end{defn}

\begin{defn}[Generalized pairs] 
	A \emph{generalized sub-pair} consists of 
	\begin{itemize}
		\item a normal variety $X$ equipped with a projective morphism $X\to Z$,
		\item an $\Rr$-divisor $B$ on $X$, and
		\item a b-$\Rr$-Cartier b-divisor over $X$ represented by some projective birational morphism $X'\xto X$ and $\Rr$-Cartier divisor $M'$ on $X'$, 
	\end{itemize}
	such that
	\begin{itemize}
		\item $M'$ is nef$/Z$, and
		\item $K_X+B+M$ is $\Rr$-Cartier, where $M$ is the strict transform of $M'$ on $X$.
	\end{itemize}
	We may say that $(X,B+M)$ is a generalized sub-pair with data $X'\xto X\to Z$ and $M'$. 
	We call $B$ the \emph{boundary part} and $M$ the \emph{nef part}. If $\dim Z=0$, the generalized sub-pair is called \emph{projective}, and we will omit $Z$. If $Z=X$ and $X\to Z$ is the identity map, we will omit $Z$. We omit the prefix ``sub" everywhere if $B\ge0$.

	Since a b-$\Rr$-Cartier b-divisor is defined birationally, in practice we will often replace $X'$ by a higher model and replace $M'$ by its pullback.
	Possibly replacing $X$ by a higher model and $M'$ by its pullback, we may assume that $f$ is a log resolution of $(X,B)$, and write
	$$K_{X'}+B'+M'=f^*(K_X+B+M)$$
	for some uniquely determined $B'$. The \emph{generalized log discrepancy} of a divisor $E$ on $X'$ with respect to $(X,B+M)$ is $1-\mult_E{B}$ and denoted by $a(E,X,B+M)$. 
	We say that $(X,B+M)$ is \emph{generalized $\epsilon$-lc} $($respectively \emph{generalized klt, generalized lc}$)$ for some non-negative real number $\epsilon$ if $a(E,X,B+M)\ge \epsilon$ $($respectively $>0,\ge0)$ for any $E$. 
\end{defn}

\begin{defn}
	Let $X\to Z$ be a contraction. We say that $X$ is of \emph{Fano type} over $Z$ if $(X,B)$ is klt and $-(K_{X}+B)$ is big and nef over $Z$ for some boundary $B$.
\end{defn}

\begin{rem}
	Assume that $X$ is of Fano type over $Z$. Then we can run the MMP$/Z$ on any $\Rr$-Cartier divisor $D$ on $X$ which terminates with some model $Y$ (c.f. \cite[Corollary 2.9]{PS09}, \cite{BCHM10}). 
\end{rem}

\begin{defn}[Generalized dlt]
	Let $(X,B+M)$ be a generalized pair with data $X'\xto X\to Z$ and $M'$. We say that $(X,B+M)$ is \emph{generalized dlt} if it is generalized lc and there is a closed subset $V\subseteq X$ such that
	\begin{enumerate}
		\item $X\setminus V$ is smooth and $B|_{X\setminus V}$ is a snc divisor, and
		\item if $a(E,X,B+M)=0$ for some prime divisor $E$ over $X$, then $\Center_X E\nsubseteq V$ and $\Center_XE\setminus V$ is a non-klt center of $(X,B)|_{X\setminus V}$.
	\end{enumerate}
\end{defn}

\begin{rem}
	If $(X,B+M)$ is a $\Qq$-factorial generalized dlt pair with data $X'\xto X\to Z$ and $M'$, then $X$ is klt.
\end{rem}

\begin{defn}[Generalized $a$-lc thresholds]
	Let $(X,B+M)$ be a generalized pair with data $X'\xto X\to Z$ and $M'$. Suppose that $(X\ni x,B+M)$ is generalized $a$-lc at $x$ for some non-negative real number $a$, that is, $a(E,X,B+M)\ge a$ for any prime divisor $E$ over $X$ with $\Center_XE=\bar{x}$, where $x\in X$ is a point with $\dim x<\dim X$. Assume that $D$ is an effective $\Rr$-Cartier divisor and $N'$ is a nef$/Z$ $\Rr$-divisor on $X'$, such that $D+N$ is $\Rr$-Cartier, where $N=f_*N'$. The \emph{generalized $a$-lc threshold of $D+N$ with respect to $(X,B+M)$ at $x$} is defined as
	\begin{align*}
	\alct(X\ni x,B+M;D+N):=\sup\{t\ge0\mid&(X\ni x,(B+tD)+(M+tN))\\&\text{ is generalized $a$-lc at }x\}.
	\end{align*}
\end{defn}
When $M'=N'=0$, we just call $\alct(X\ni x,B;D)$ the $a$-lc threshold of $D$ with respect to $(X,B)$ at $x$.

\subsection{MMP for generalized pairs}
For generalized pairs, one can ask whether one can run MMP and whether it terminates. However the MMP for generalized pairs is not completely established, but some important cases could be derived from the standard MMP. We elaborate some results which are developed in \cite[\S 4]{BZ16}.

Let $(X,B+M)$ be a $\Qq$-factorial generalized pair with data $X'\xto X\to Z$ and $M'$, and $A$ a general ample$/Z$ divisor on $X$. Moreover, assume that

$ (\star)$  for any $0<\epsilon\ll1$, there exists a boundary $\Delta_{\epsilon}\sim_{\Rr,Z}B+M'+\epsilon A$, such that $(X,\Delta_{\epsilon})$ is klt.

Under assumption $(\star)$, we can run a generalized MMP$/Z$ on $(K_X+B+M)$ with scaling of $A$, although the termination is not known (c.f. \cite[\S 4]{BZ16}).

The following lemma shows that assumption $(\star)$ is satisfied in two cases.    
\begin{lem}[{\cite[Lemma 3.5]{HL18}}]
	Let $(X,B+M)$ be a generalized lc pair with data $X'\xto X\to Z$ and $M'$, and $A$ an ample$/Z$ divisor, such that either
	\begin{enumerate}
		\item $(X,B+M)$ is generalized klt, or
		\item $(X,C)$ is klt for some boundary $C$.
	\end{enumerate}   
	Then there exists a boundary $\Delta\sim_{\Rr,Z}B+M+A$ such that $(X,\Delta)$ is klt. Moreover, if $X$ is $\Qq$-factorial, we may run a generalized MMP$/Z$ on $(K_X+B+M)$.
\end{lem}
We need the following results.

\begin{lem}[{\cite[Lemma 4.4]{BZ16}}]
	Let $(X,B+M)$ be a generalized lc pair with data $(X'\xto X\to X,M')$ such that $K_{X}+B+M$ is not pseudo-effective$/Z$, and either
	\begin{enumerate}
		\item $(X,B+M)$ is generalized klt, or
		\item $(X,C)$ is klt for some boundary $C$.
	\end{enumerate}
	Then any generalized MMP on $(K_{X}+B+M)$ with scaling of some ample$/Z$ $\Rr$-Cartier divisor terminates with a Mori fiber space.
\end{lem}

\begin{lem}[Generalized dlt modification {\cite[Proposition 3.9]{HL18}}]
	Let $(X,B+M)$ be a generalized pair with data $X'\xto X\to Z$ and $M'$. Then possibly replacing $X'$ by a higher model, there exist a $\Qq$-factorial generalized dlt pair $(Y,B_Y+M_Y)$ with data $X'\stackrel{g}\to Y\to Z$ and $M'$, and a contraction $\phi:Y\to X$ such that $K_Y+B_Y+M_Y=\phi^*(K_X+B+M)$. Moreover, each exceptional divisor of $\phi$ is a component of $\lf B_Y\rf$. We call $(Y,B_Y+M_Y)$ a generalized dlt modification of $(X,B+M)$.
\end{lem}

\subsection{Complements}
We now introduce complements for generalized pairs. Note that our definition of complements is a slightly generalization of the usual definition of complements for generalized pairs (c.f. \cite{Bir19}).
\begin{defn}[Complements]\label{defn:complements}
	Let $(X,B+M)$ be a generalized pair with data $X'\xto X\to Z$ and $M'=\sum \mu_jM_j'$ such that $\mu_j\ge0$ and $M_j'$ are b-Cartier nef$/Z$ divisors. We say that $(X/Z\ni z,B^++M^+)$ is an \emph{$\Rr$-complement} of $(X/Z\ni z,B+M)$ if $(X,B^++M^+)$ is generalized lc, $B^+\ge B,\mu_j^+\ge\mu_j$ for any $j$, and $K_X+B^++M^+\equiv0$ over a neighborhood of $z$, where $M^{+'}=\sum_j\mu_j^+M_j'$. In addition, if $(X/Z\ni z,B^++M^+)$ is generalized $\epsilon$-lc over $z$, we say that $(X/Z\ni z,B^++M^+)$ is an $(\epsilon,\Rr)$-complement of $(X/Z\ni z,B+M)$.
	
	Let $n$ be a positive integer. 
	We say that a generalized pair $(X/Z\ni z,B^++M^+)$ is an \emph{$n$-complement} of $(X/Z\ni z,B+M)$, if over a neighborhood of $z$, we have
	\begin{enumerate}
		\item $(X,B^++M^+)$ is generalized lc,
		\item $n(K_X+B^++M^+)\sim 0$,
		\item $nB^+\ge n\lfloor B\rfloor+\lf(n+1)\{B\}\rf$, and
		\item $n\mu_j^+\ge n\lfloor \mu_j\rfloor+\lf(n+1)\{\mu_j\}\rf$ and $nM^{+'}$ is b-Cartier, where $M^{+'}=\sum \mu_j^+M_j'$.
	\end{enumerate}
	We say that $(X/Z\ni z,B^++M^+)$ is a monotonic $n$-complement of $(X/Z\ni z,B+M)$ if additionally we have $B^+\ge B$ and $\mu^+_j\ge\mu_j$ for any $j$. 
	If $(X/Z\ni z,B^++M^+)$ is generalized $\epsilon$-lc over a neighborhood of $z$, we say that $(X/Z\ni z,B^++M^+)$ is an $(\epsilon,n)$-complement of $(X/Z\ni z,B)$.
\end{defn}

We say that $(X/Z\ni z,B+M)$ is \emph{$(\epsilon,\Rr)$-complementary} (respectively \emph{$(\epsilon,n)$-complementary}) if it has an $(\epsilon,\Rr)$-complement (respectively $(\epsilon,n)$-complement).

If $\dim Z=0$, we will omit $Z$ and $z$.
If for any  $z\in Z$, $(X/Z\ni z,B^++M^+)$ is $(\epsilon,\Rr)$-complementary (respectively $(\epsilon,n)$-complementary), then we say that $(X/Z,B+M)$ is $(\epsilon,\Rr)$-complementary (respectively $(\epsilon,n)$-complementary).

The following lemma is well-known to experts (c.f. \cite[Lemma 3.13]{CH20}). We will use the lemma frequently without citing it in this paper. 
\begin{lem}\label{lem: pullbackcomplements}
	Let $\epsilon$ be a non-negative real number, $(X,B+M)$ be a generalized pair with data $X'\xto X\to Z$ and $M'$. Assume that $g:X\dashrightarrow X''$ is a birational contraction and $B'',M''$ are the strict transforms of $B,M$ on $X'.$
	\begin{enumerate}
		\item If $(X/Z\ni z,B+M)$ is $(\epsilon,\Rr)$-complementary, then $(X''/Z\ni z,B''+M'')$ is $(\epsilon,\Rr)$-complementary.
		\item Let $n$ be a positive integer. If $g$ is $-(K_X+B+M)$-non-positive and $(X''/Z\ni z,B''+M'')$ is $(\epsilon,\Rr)$-complementary (respectively monotonic $(\epsilon,n)$-complementary), then $(X/Z\ni z,B+M)$ is $(\epsilon,\Rr)$-complementary (respectively monotonic $(\epsilon,n)$-complementary).
	\end{enumerate}
\end{lem}

\subsection{Bounded families}
\begin{defn}
	A \emph{couple} consists of a normal projective variety $X$ and a divisor $D$ on $X$ such that $D$ is reduced. Two couples $(X, D)$ and $(X', D')$ are isomorphic if there exists an isomorphism $X\to X'$ mapping $D$ onto $D'$. A set $\mathcal{P}$ of couples is \emph{bounded} if there exist finitely many projective morphisms $V^{i}\to U^{i}$ of varieties and reduced divisors $C^{i}$ on $V^{i}$ such that for each $(X,D)\in\mathcal{P}$, there exists $i$, a closed point $t \in U^{i}$, and two couples $(X,D)$ and $(V_{t}^{i}, C_{t}^{i})$ are isomorphic, where $V_{t}^{i}$ and $C_{t}^{i}$ are the fibers over $t$ of the morphisms $V^i\to U^{i}$ and $C^{i}\to U^{i}$ respectively.
	
	A set $\mathcal{C}$ of projective pairs $(X,B)$ is said to be \emph{log bounded} if the set of the corresponding set of couples $\{(X, \Supp B)\}$ is bounded. A set $\mathcal{D}$ of projective varieties $X$ is said to be bounded if the corresponding set of couples $\{(X, 0)\}$ is bounded. A log bounded (resp. bounded) set is also called a \emph{log bounded family} (resp. \emph{bounded family}).
\end{defn}
We will need the following theorem.

\begin{thm}[BBAB Theorem, {\cite[Theorem 1.1]{Bir16b}}]\label{thm:bab}
	Let $d$ be a positive integer and $\epsilon$ a positive real number. Then the projective varieties $X$ such that
	\begin{enumerate}
		\item $(X,B)$ is $\epsilon$-lc of dimenison $d$ for some boundary $B$, and
		\item $-(K_X+B)$ is nef and big,
	\end{enumerate}
	form a bounded family.
\end{thm}

The following lemma is an easy consequence of BBAB Theorem.

\begin{lem}\label{lem:babforcomplementarypair}
	Let $d$ be a positive integer and $\epsilon$ a positive real number. Then the projective varieties $X$ such that
	\begin{enumerate}
		\item $\dim X=d$,
		\item $X$ is of Fano type, and
		\item $(X,B+M)$ is an $(\epsilon,\Rr)$-complementary projective generalized pair with data $X'\xto X$ and $M'$, for some boundary $B$ and $M'$,
	\end{enumerate}
	form a bounded family.
\end{lem}

\begin{proof}
	Since $X$ is of Fano type, there exists a boundary $C$ such that $(X,C)$ is klt and $-(K_X+C)$ is ample. It follows that $(X,D+\frac{M}{2})$ is generalized $\frac{\epsilon}{2}$-lc, and $-(K_X+D+\frac{M}{2})$ is ample, where $D=\frac{B+C}{2}$. 
	Let $A\sim_\Rr -(K_X+D+\frac{M}{2})/2$ be a general ample $\Rr$-divisor such that $(X,D+A+\frac{M}{2})$ is generalized $\frac{\epsilon}{2}$-lc. We may write
	$$K_{X'}+D'+\frac{M'}{2}+f^*A=f^*(K_X+D+\frac{M}{2}+A).$$
	As $\frac{M'}{2}+f^*A$ is nef and big, there exists an effective $\Rr$-divisor $E$ such that for each positive integer $k$, $\frac{M'}{2}+f^*A\sim_\Rr E/k+A_k$ for some general ample $\Rr$-divisor $A_k$. We may choose $k$ sufficiently large such that $(X',D'_k)$ is sub-$\frac{\epsilon}{4}$-lc, where $D'_k=D'+E/k+A_k$. Let $K_X+D_k=f_*(K_{X'}+D'_k)$. Then $(X,D_k)$ is $\frac{\epsilon}{4}$-lc and 
	\begin{align*}
	-(K_X+D_k)&\sim_\Rr -f_*(K_{X'}+D'+\frac{M'}{2}+f^*A)\\&\sim_\Rr -(K_X+D+\frac{M}{2}+A)\sim_\Rr A
	\end{align*}
	is ample. According to Theorem \ref{thm:bab}, $X$ belongs to a bounded family.
\end{proof}

\section{Uniform linearity of minimal log discrepancies}\label{section3}
\subsection{Uniform linearity of MLDs}
In \cite{CH20}, the first author and Han show the uniform linearity of minimal log discrepancies for surface pairs with positive mlds (See \cite[Theorem 4.8]{CH20}). Indeed, one can show the linearity of minimal log discrepancies for generalized pairs with some extra conditions, that is, we have the following conditions.
\begin{thm}\label{thm:mldlinearity}
	Let $\epsilon$ be a non-negative real number, $d,c,m,l$ positive integers, $\bm{r}_0:=(r_1,\dots,r_c)\in\Rr^c$ a point such that $r_0:=1,r_1,\dots,r_c$ are linearly independent over $\Qq$, and $s_1,\ldots, s_{m+l}: \Rr^{c}\to\Rr$ $\Qq$-linear functions. Then there exist a positive real number $\delta$ and a $\Qq$-linear function $f(\bm{r}):\Rr^c\to\Rr$ depending only on $\epsilon,d,c,m,l,\bm{r}_0$ and $s_i$ satisfying the following.
	
	Assume that $(X,B(\bm{r}_0)+M(\bm{r}_0))$ is a projective generalized pair with data $X'\xto X$ and $M'(\bm{r}_0):=\sum_{j=1}^{l}s_{m+j}(\bm{r}_0)M_j'$ such that
	\begin{itemize}
		\item $\dim X$=d,
		\item $X$ is of Fano type,
		\item $B(\bm{r}):=\sum_{i=1}^ms_i(\bm{r})B_i,$ and $B_i\ge0$ are Weil divisors,
		\item $M_j'$ are b-Cartier nef divisors, and
		\item $(X,B(\bm{r}_0)+M(\bm{r}_0))$ is $(\epsilon,\Rr)$-complementary.
	\end{itemize}
	Then the following hold:
	\begin{enumerate}
		\item $f(\bm{r}_0)\ge\epsilon$, and if $\epsilon\in\Qq$, then $f(\bm{r})=\epsilon$ for any $\bm{r}\in\Rr^c$.
		
		\item Suppose that $\epsilon$ is positive, then there exists a prime divisor $E$ over $X$ such that
		$$\mld(X,B(\bm{r})+M(\bm{r}))=a(E,X,B(\bm{r})+M(\bm{r}))\ge f(\bm{r})$$
		for any $\bm{r}\in\Rr^c$ satisfying $||\bm{r}-\bm{r}_0||_\infty\le\delta$. 
		
		\item $(X,B(\bm{r})+M(\bm{r}))$ is $(f(\bm{r}),\Rr)$-complementary for any $\bm{r}\in\Rr^c$ satisfying $||\bm{r}-\bm{r}_0||_\infty\le\delta$.
	\end{enumerate}
\end{thm}
\begin{proof}
	By \cite[Theorem 3.5]{Chen20}, there exists a positive real number $\delta_0$ depending only on $d,c,m$ and $\bm{r}_0$ such that $(X,B(\bm{r})+M(\bm{r}))$ is lc for any $\bm{r}\in\Rr^c$ satisfying $||\bm{r}-\bm{r}_0||_{\infty}\le\delta_0.$ In particular, there exist a finite set $\Gamma_1=\{a_i\}_{i=1}^{k}$ of positive real numbers and a finite set $\Gamma_2$ of non-negative rational numbers depending only on $d,c,m,s_i,\bm{r}_0$ and $\delta_0$ such that
	$$K_X+B(\bm{r}_0)+M(\bm{r}_0)=\sum_{i=1}^{k} a_i(K_X+B^i+M^i),$$ 
	and $(X,B^i+M^i)$ is lc for some $B^i\in\Gamma_2$. Moreover, by \cite[Theorem 1.1]{Chen20} and \cite[Lemma 2.24]{Bir19}, there exists a positive integer $I_0$ such that $I_0(K_X+B_i+M_i)$ is Cartier. Then by the same arguments as in \cite[Lemma 4.7]{CH20} and \cite[Theorem 3.17]{Chen20}, one can find a positive real number $\delta$ with the required properties.
\end{proof}

\subsection{Uniform rational polytopes}

\begin{thm}\label{thm: uniformpolytopeforepsiloniRcomp}
	Let $\epsilon$ be a non-negative real number, $d,m$ and $l$ positive integers, and $\bm{v}_0=(v_1^0,\ldots,v_{m+l}^0)\in\Rr^{m+l}$ a point. Then there exist positive real numbers $a_k$, non-negative real numbers $\epsilon_k$ and points $\bm{v}_k=(v_1^k,\ldots,v_{m+l}^k)\in\Rr^{m+l}$ depending only on $\epsilon,d,m,l$ and $\bm{v}_0$ satisfying the following.   
	\begin{enumerate}
		\item $\sum a_k=1,\sum a_k\bm{v}_k=\bm{v}_0$, $\sum a_k\epsilon_k\ge\epsilon$. Moreover, if $\epsilon\in \Qq$, then we may pick $\epsilon_k=\epsilon$ for any $k$.
		\item Assume that $(X,(\sum_{i=1}^m v_i^0B_{i})+(\sum_{j=1}^{l}v_{m+j}^0M_j))$ is a generalized pair with data $X'\xto X$ and $\sum_{j=1}^{l}v_{m+j}^0M_j'$ such that
		\begin{itemize} 
			\item $\dim X=d$, 
			\item $X$ is of Fano type,
			\item $B_1,\ldots,B_m\ge0$ are Weil divisors on $X$,
			\item $M_j'$ is b-Cartier nef for any $1\le j\le l$, and
			\item $(X,(\sum_{i=1}^m v_i^0B_{i})+(\sum_{j=1}^{l}v_{m+j}^0M_j))$ is $(\epsilon,\Rr)$-complementary.
		\end{itemize}
		Then $(X,(\sum_{i=1}^m v_i^kB_i)+(\sum_{j=1}^{l}v_{m+j}^kM_j))$ is $(\epsilon_k,\Rr)$-complementary for any $k$.
	\end{enumerate}
\end{thm}

\begin{proof}
	The result follows from Theorem \ref{thm:mldlinearity} and the same argument as in \cite[Theorem 5.6]{HLS19}.
\end{proof}

\section{Proof of Theorem \ref{thm: (n,I)complforfiniterat}}\label{section4}

\subsection{From DCC sets to finite sets}
The proofs of Theorem \ref{thm: dcc limit lc} and Theorem \ref{thm: dcc limit lc complementary} are very similar like the proofs of \cite[Theorem 5.18 and Theorem 5.19]{HLS19}, for readers' convenience, we give proofs here.
\begin{thm}\label{thm: dcc limit lc}
	Let $\epsilon$ be non-negative real numbers, $d$ a positive integers, and $\Ii\subseteq [0,1]$ a DCC set. Then there exist a finite set $\Ii'\subseteq \bar\Ii$, and a projection $g:\bar\Ii\to \Ii'$ depending only on $\epsilon,d$ and $\Ii$ satisfying the following. 
	
	Assume that $(X,\sum b_iB_i)$ is a pair such that
	\begin{itemize}
		\item $\dim X=d$, 
		\item $X$ is of Fano type,
		\item $B_i\ge0$ is a $\Qq$-Cartier Weil divisor and $b_i\in\Ii$ for any $i$,
		and
		\item $(X,\sum b_iB_i)$ is $(\epsilon,\Rr)$-complementary.
	\end{itemize}
	Then
	\begin{enumerate}
		\item  $g(\gamma)\ge \gamma$ for any $\gamma\in \Ii$,
		\item $g(\gamma')\ge g(\gamma)$ for any $\gamma'\ge \gamma$, $\gamma,\gamma'\in\Ii$, and
		\item  $(X,\sum g(b_i)B_i)$ is $\epsilon$-lc.
	\end{enumerate}
\end{thm}



\begin{proof}
	According to \cite[Theorem 5.20]{HLS19}, we only need to show the case when $\epsilon$ is a positive real number.
	We first show that $(X, B=\sum b_iB_i)$ belongs to a log bounded family. Let $c=\min\{\alpha>0\mid\alpha\in\Gamma\}$. 
	By Lemma \ref{lem:babforcomplementarypair}, $X$ belongs to a bounded family. In particular, there exists a very ample divisor $A$ on $X$ and a positive real number $r$ which only depends on $d$ and $\epsilon$ such that $A^{d-1}\cdot(-K_X) \le r$. Since
	\begin{align*}
	\Supp B\cdot A^{d-1}&=\frac{1}{c}(K_X+c\sum B_i-K_X)\cdot A^{d-1}\\&\le\frac{1}{c}(K_X+\sum b_iB_i-K_X)\cdot A^{d-1}\le\frac{r}{c},
	\end{align*}
	$(X, \sum B_i)$ is log bounded by \cite[Lemma 2.20]{Bir19}.
	
	Let $(\mathcal{X}, \mathcal{D}) \to U$ be the corresponding log bounded family. Then there exists a stratification $U_1,\dots,U_l$ of $U$ such that, possibly after taking a finite \'{e}tale cover, each restricted family $(\mathcal{X}_{U_i}, \mathcal{D}_{U_i}) \to U_i$ admits a fiberwise log resolution.
	
	Let $\mathcal{S}$ be the set of pairs $(X,B)$ satisfying the conditions of the theorem, and
	$$
	\Gamma'':= \{\epsilon\text{-}\lct(X, \sum_{i=1}^{j-1} b_iB_i + \sum_{i=j+1}^s b_i B_i; B_{j})\mid (X,\sum_{i=1}^s b_iB_i)\in\mathcal{S},1\le j\le s\}.
	$$
	We claim that $\Gamma''$ satisfies the ACC. Since we can take a fiberwise log resolution of the bounded family $(\mathcal{X}_{U_i}, \mathcal{D}_{U_i}) \to U_i$, for $(X, B)$ which is a fiber of this family, the induced resolution $f: Y \to (X,B)$ is well behaved in the following sense. We may write 
	$$f^* K_X =K_Y+ \sum_k a_kE_k\ \ \text{and}\ \ f^* B_i = \sum_k a_{ik}E_k$$ 
	for some real numbers $a_k=a(E_k,X,0)$ and $a_{ik}\ge0$, then $a_k$ and $a_{ik}$ lie in a finite set which only depends on the bounded family. Hence $m_{jk}:=\mult_{E_k} (f^*(K_X+\sum_{i=1}^{j-1} b_iB_i + \sum_{i=j+1}^s b_i B_i)-K_Y)$ lie in a DCC set which only depends on the DCC set $\Gamma$ and the bounded family. Note that $\epsilon\text{-}\lct(X, \sum_{i=1}^{j-1} b_iB_i + \sum_{i=j+1}^s b_i B_i; B_{j}) = \min_k\{\frac{1-\epsilon-m_{jk}}{a_{jk}}\}$. Therefore it belongs to an ACC set.
	
	Next applying \cite[Lemma 5.17]{HLS19}, there exists a finite set $\Gamma\pr \subset \overline{\Gamma}$ and a projection $g: \overline{\Gamma} \to \Gamma\pr$ satisfying \cite[Lemma 5.17]{HLS19}(1)-(3).
	
	It remains to show (3). Suppose on the contrary that there exists some $0 \le j \le s-1$, such that $(X, \sum_{i=1}^j g(b_i) B_i + \sum_{i =j+1}^s b_i B_i)$ is $\epsilon$-lc, but $(X, \sum_{i=1}^{j+1} (g(b_i) B_i + \sum_{i =j+2}^s b_i B_i)$ is not. Let
	$$
	\beta = \epsilon\text{-}\lct(X, \sum_{i=1}^j g(b_i) B_i + \sum_{i =j+2}^s b_i B_i; B_{j+1}).
	$$
	Then $b_{j+1} \le \beta < g(b_{j+1})$. Note that $\beta \in \Gamma\prpr$, but this contradicts to \cite[Lemma 5.17]{HLS19}(3).
\end{proof}

\begin{thm}\label{thm: dcc limit lc complementary}
	Let $\epsilon$ be non-negative real numbers, $d$ a positive integers, and $\Ii\subseteq [0,1]$ a DCC set. Then there exist a finite set $\Ii'\subseteq \bar\Ii$, and a projection $g:\bar\Ii\to \Ii'$ depending only on $\epsilon,d$ and $\Ii$ satisfying the following. 
	
	Assume that $(X,\sum b_iB_i)$ is a pair such that
	\begin{itemize}
		\item $\dim X=d$, 
		\item $X$ is of Fano type,
		\item $B_i\ge0$ is a $\Qq$-Cartier Weil divisor and $b_i\in\Ii$ for any $i$,
		and
		\item $(X,\sum b_iB_i)$ is $(\epsilon,\Rr)$-complementary.
	\end{itemize}
	Then
	\begin{enumerate}
		\item  $g(\gamma)\ge \gamma$ for any $\gamma\in \Ii$,
		\item $g(\gamma')\ge g(\gamma)$ for any $\gamma'\ge \gamma$, $\gamma,\gamma'\in\Ii$, and
		\item  $(X,\sum g(b_i)B_i)$ is $(\epsilon,\Rr)$-complementary.
	\end{enumerate}
\end{thm}

\begin{proof}
	We first show that there exists a finite set $\Gamma' \subseteq \Gamma$ and a projection $g: \overline{\Gamma} \to \Gamma'$ satisfying (1), (2) and that $-(K_X+\sum g(b_i)B_i)$ is pseudo-effective. Suppose it does not hold. By Theorem \ref{thm: dcc limit lc}, there exists a sequence of pairs $(X_k, B_{(k)}:=\sum_i b_{k, i}B_{k, i})$ of dimension $d$ and a sequence of projections $g_k: \overline{\Gamma} \to \overline{\Gamma}$ such that for any $k, i$, $b_{k, i} \in \Gamma$, $B_{k, i} \ge 0$ is a $\QQ$-Cartier Weil divisor, $(X_k, B_{(k)})$ is $(\epsilon,\Rr)$-complementary, $(X_k, B'_{(k)}):=\sum_i g_k(b_{k,i})B_{k,i})$ is $\epsilon$-lc, $b_{k,i} + \frac{1}{k} \ge g_k(b_{k,i}) \ge b_{k,i}$, and $-(K_{X_k}+B'_{(k)})$ is not pseudo-effective.
	
	We may run a $-(K_{X_k}+B'_{(k)})$-MMP with scaling of an ample divisor and reach a Mori fiber space $Y_k \to Z_k$, such that $-(K_{Y_k}+B'_{(Y_k)})$ is anti-ample over $Z_k$, where $B'_{(Y_k)}$ is the strict transform of $B'_{(k)}$ on $Y_k$. Since $-(K_{X_k}+B_{(k)})$ is pseudo-effective, $-(K_{Y_k}+B_{(Y_k)})$ is nef over $Z_k$, where $B_{(Y_k)}$ is the strict transform of $B_{(k)}$ on $Y_k$.
	
	For each $k$, there exists a positive integer $k_j$ and a positive real number $b_k^+$, such that $b_{k, k_j} \le b_k^+ < g_k(b_{k, k_j})$ and $K_{F_k} + B^+_{F_k} \equiv 0$, where
	$$
	K_{F_k}+B_{F_k}^+ := \left.\left(K_{Y_{k}}+(\sum_{i<k_{j}} g_{k}\left(b_{k, i}\right) B_{Y_{k}, i})+b_{k}^{+} B_{Y_{k}, k_{j}}+(\sum_{i>k_{j}} b_{k, i} B_{Y_{k}, i})\right)\right|_{F_{k}},
	$$
	$B_{Y_k,i}$ is the strict transform of $B_{k,i}$ on $Y_k$ for any $i$, and $F_k$ is a general fiber of $Y_k \to Z_k$. Since $(X_k, B_{(k)})$ is $(\epsilon, \mathbb{R})$-complementary, $(Y_k, B_{(Y_k)})$ is lc. Thus by Theorem \ref{thm: dcc limit lc} $(Y_k, B'_{Y_k})$ is lc. Therefore $(F_k, B^+_{F_k})$ is lc.
	
	Since $g_k(b_{k, k_j})$ belongs to the DCC set $\overline{\Gamma}$ for any $k, k_j$, possibly passing to a subsequence, we may assume that $g_k(b_{k, k_j})$ is non-decreasing. Since $g_k(b_{k, k_j}) - b_k^+ >0$ and
	$$
	\lim\limits_{k \to +\infty}(g_k(b_{k, k_j}) - b_k^+) = 0,
	$$
	possibly passing to a sebsequence we may assume that $b^+_k$ is strictly increasing.
	
	Now $K_{F_k} + B_{F_k}^+ \equiv 0$, the coefficients of $B^+_{F_k}$ belong to the DCC set $\overline{\Gamma} \cup \{b^+_k\}_{k=1}^\infty$. Since $b_k^+$ is strictly increasing, this contradicts to the global ACC \cite[Theorem 1.4]{HMX14}.
	
	It remains to show that $(X,\sum g(b_i)B_i)$ is $(\epsilon,\Rr)$-complementary. As $X$ is of Fano type, we may run a $-(K_X+\sum g(b_i) B_{i})$-MMP and get a model $Y$ on which $-(K_Y+\sum g(b_i) B_{Y,i})$ semiample, where $B_{Y,i}$ is the strict transform of $B_i$ on $Y$ for any $i$. Since $(X,\sum b_i B_i)$ is $(\epsilon, \Rr)$-complementary, so is $(Y, \sum b_i B_{Y,i})$, and in particular $(Y, \sum b_i B_{Y,i})$ is $\epsilon$-lc. By the construction $(Y+\sum g(b_i) B_{Y,i})$ is also $\epsilon$-lc. Since $-(K_Y+\sum g(b_i) B_{Y,i})$ is semiample, we can see that $(Y, \sum g(b_i) B_{Y,i})$ is $(\epsilon, \Rr)$-complementary and thus $(X,\sum g(b_i)B_i)$ is also $(\epsilon, \Rr)$-complementary.
\end{proof}

\subsection{Proof of Theorem \ref{thm: (n,I)complforfiniterat}}

We first show that Theorem \ref{thm: conj for global FT} holds when $\Gamma\subseteq[0,1]\cap\Qq$ is a finite set.
\begin{prop}\label{prop1}
	Theorem \ref{thm: conj for global FT} holds when $\Ii\subseteq[0,1]\cap\Qq$ is a finite set. Moreover, $(X,B+M)$ has a monotonic $(\epsilon,n)$-complement $(X,B^++M)$.
\end{prop}
\begin{proof}
	We may run a $-(K_X+B+M)$-MMP which terminates with a model $X''$ such that $-(K_{X''}+B''+M'')$ is nef, where $B''$ and $M''$ are the strict transforms of $B$ and $M$ on $X''$ respectively. Note that since $(X,B+M)$ is $(\epsilon,\Rr)$-complementary, $(X'',B''+M'')$ is generalized $\epsilon$-lc. Possibly replacing $(X,B+M)$ by $(X'',B''+M'')$, we may assume that $-(K_X+B+M)$ is nef.
	
	By Lemma \ref{lem:babforcomplementarypair}, $X$ belongs to a bounded family. By \cite[Lemma 2.25]{Bir19}, there exists an integer $n_0$ such that $-n_0(K_X+B+M)$ is Cartier. Hence $-n(K_X+B+M)$ is base point free for some positive integer $n$ which only depends on $n_0$ and $d$ by \cite{Kol93}. It implies that $(X,B+M)$ has a monotonic $(\epsilon,n)$-complement $(X,B^++M)$ for some $\Rr$-divisor $B^+\ge B$.    
\end{proof}

\begin{proof}[Proof of Theorem \ref{thm: (n,I)complforfiniterat}]
	We may assume that $1\in\Ii$. Possibly replacing $(X,B+M)$ by a generalized dlt modification, we may assume that $X$ is $\Qq$-factorial. In the case when $M'\equiv 0$, we can apply Theorem \ref{thm: dcc limit lc complementary} and assume that $\Ii$ is a finite set. Therefore, it suffices to prove the case when $\Gamma$ is finite. 
	
	By Theorem \ref{thm: uniformpolytopeforepsiloniRcomp}, there exist two finite sets $\Ii_0\subseteq(0,1],\Ii_1\subseteq[0,1]\cap\Qq$ and non-negative real numbers $\epsilon_i$ depending only on $d$ and $\Ii$ such that $(X,B_i+M_{(i)})$ is $(\epsilon_i,\Rr)$-complementary, $\sum a_i=1,\sum a_i\epsilon_i\ge\epsilon$ and
	$$K_X+B+M=\sum a_i(K_X+B_i+M_{(i)})$$
	for some $a_i\in\Ii_0,B_i\in\Ii_1$ and $M_{(i)}'=\sum_j \mu_{ij}M_{j}'$ with $\mu_{ij}\in\Gamma_1$. Moreover, if $\Ii\subseteq\Qq$, then we may pick $\Ii_0=\{1\},$ $B_i=B$ and $M_{(i)}=M$.
	
	For each $i$, we may run an MMP on $-(K_X+B_i+M_{(i)})$ and terminates with a model $Y_i$, such that $-(K_{Y_i}+B_{Y_i,i}+M_{Y_i,i})$ is nef, where $B_{Y_i,i}$ and 
	$M_{Y_i,i}$ are the strict transforms of $B_i$ and $M'_{(i)}$ on $Y_i$ respectively. Since $(X,B+M)$ is $(\epsilon,\Rr)$-complementary, $(Y_i,B_{Y_i}+M_{Y_i})$ is generalized $\epsilon$-lc, where $B_{Y_i}$ and $M_{Y_i}$ are the strict transforms of $B$ and $M'$ on $Y_i$ respectively. According to the construction of $\Ii_1$, $(Y_i,B_{Y_i,i}+M_{Y_i,i})$ is generalized $\epsilon_i$-lc. By Proposition \ref{prop1}, there exists a positive integer $n$ depending on $d$ and $\Ii_1$, such that $(Y_i,B_{Y_i,i}+M_{Y_i,i})$ has a monotonic $(\epsilon_i,n)$-complement $(Y_i,B_{Y_i,i}^++M_{Y_i,i}))$.  
	
	By Lemma \ref{lem: pullbackcomplements}, $(X,B_{i}+M_{(i)})$ has a monotonic $(\epsilon_i,n)$--complement $(X,(B_{i}+G_i)+M_{(i)})$ for some $G_{i}\ge0$. Let $B^{+}:=\sum a_i(B_{i}+G_i)$. Then $(X,B^++M)$ is an $(\bm{\epsilon},n,\Gamma_0)$-decomposable 
	$(\epsilon,n)$-complement of $(X,B+M)$. We finish the proof.
\end{proof}

\section{Existence of $(\epsilon,n)$-complements}\label{section5}

We propose the following conjecture, which is a generalization of Theorem \ref{thm: (n,I)complforfiniterat}.

\begin{conj}\label{conj: (n,I)complforfiniterat}
	Let $d$ a positive integers, $\epsilon$ a positive real number and $\Ii\subseteq [0,1]$ a DCC set. Then there exist a positive integer $n$, non-positive real numbers $\epsilon_i$ and a finite set $\Ii_0\subseteq(0,1]$ depending only on $d,\epsilon$ and $\Ii$ satisfying the following. 
	
	Assume that $(X,B+M)$ is a generalized pair with data $X'\xto X\to Z$ and $M'=\sum \mu_jM_j'$ such that
	\begin{enumerate}
		\item $\dim X=d$,	
		\item $X$ is of Fano type of $Z$,
		\item $B\in\Ii$,
		\item $M_j'$ is a b-Cartier nef$/Z$ divisor and $\mu_j\in\Ii$ for any $j$, and
		\item $(X/Z,B+M)$ is $(\epsilon,\Rr)$-complementary.
	\end{enumerate}
	Then for any point $z\in Z$, there is an $(\bm{\epsilon},n,\Ii_0)$-decomposable $(\epsilon,\Rr)$-complement $(X/Z\ni z,B^++M^+)$ of $(X/Z\ni z,B+M)$. Moreover, if $\bar{\Ii}\subseteq\Qq$, then we may pick $\Ii_0=\{1\}$, and $(X/Z\ni z,B^++M^+)$ is a monotonic $(\epsilon,n)$-complement of $(X/Z\ni z,B+M)$.
\end{conj}

We also refer readers to \cite[Problem 7.7]{CH20} for a more general formulation for the usual pairs. As we remarked, when $\epsilon=0$, it was proved in \cite{Chen20}, and also see \cite{HLS19} for the usual pairs.

\begin{prop}\label{prop:conjtoconj}
	Conjecture \ref{conj: (n,I)complforfiniterat} implies Conjecture \ref{conj a complement}.
\end{prop}

\begin{proof}
	By assumption, there exists a positive integer $n_0$, two finite sets $\Ii_0\subseteq(0,1]$ and $\mathcal{S}=\{\epsilon_i\}$ of non-negative real numbers depending only on $d$ and $\Ii$, such that possibly shrinking $Z$ near $z$, $(X/Z,B+M)$ has an $(\bm{\epsilon},n_0,\Ii_0)$-decomposable $(\epsilon,\Rr)$-complement $(X/Z,\tilde{B}+\tilde{M})$. In particular, there exist $a_i\in\Ii_0,\epsilon_i\in\mathcal{S}$ and boundaries $\tilde{B}_i'$ with nef parts $\tilde{M}_i'$ such that $\sum a_i=1,\sum a_i\epsilon_i\ge\epsilon,$  $(X/Z,\tilde{B}_i+\tilde{M}_i)$ is an $(\epsilon_i,n_0)$-complement of itself for any $i$, and
	$$K_X+\tilde{B}+\tilde{M}=\sum a_i(K_X+\tilde{B}_i+\tilde{M}_i).$$
	By \cite[Lemma 6.2]{CH20}, there exists a positive integer $n$ divisible by $pn_0$ depending only on $\epsilon,p,n_0,\Ii_0$ and $\mathcal{S}$ such that there exist positive rational numbers $a_i'$ with the following properties:
	\begin{itemize}
		\item $\sum a_i'=1$, $\sum a_i'\epsilon_i\ge\epsilon$,
		\item $n\bm{v}\in\Zz^s$, and $na_i'\in n_0\Zz$ for any $i$,
		\item $nB^+\ge n\lfloor \tilde{B}\rfloor+\lfloor (n+1)\{\tilde{B}\}\rfloor$, where $B^+:=\sum a_i' \tilde{B}_i'$, and
		\item $n\sum_{i} a_i'\mu_{ij}=n\lfloor \sum_{i} a_i\mu_{ij} \rfloor+\lfloor (n+1)\{\sum_{i} a_i\mu_{ij}\}\rfloor$ for any $j$, where $\tilde{M}_i'=\sum_j\mu_{ij}M_j'$.
	\end{itemize}
	Let $\mu_j^+=\sum_i a_i'\mu_{ij}$ for any $j$, and $M^{+'}=\sum_i a_i'\tilde{M}_i'=\sum_j\mu_j^+M_j'$, then
	\begin{align*}
	n(K_{{X}}+B^++M^+)&=n\sum a_i'(K_{X}+\tilde{B}_i+\tilde{M}_i)
	\\&=\sum \frac{a_i'n}{n_0}\cdot n_0(K_{X}+\tilde{B}_i+\tilde{M}_i)\sim_Z0.
	\end{align*}
	We conclude that $(X/Z,B^++M^+)$ is an $(\epsilon,n)$-complement of $(X/Z,B+M)$, since $\tilde{B}\ge B$ and $\tilde{\mu}_j\ge\mu_j$ for any $j$.
	
	Moreover, if $\Span_{\Qq_{\ge0}}(\bar{\Ii}\backslash\Qq)\cap (\Qq\backslash\{0\})=\emptyset$, then we may pick ${B}^+\ge B'\ge B$ and $\mu_j^+\ge\mu_j$ by \cite[Lemma 6.2]{CH20} and \cite[Lemma 6.3]{HLS19}.
\end{proof}

\begin{proof}[Proof of Theorem \ref{thm: conj for global FT}]
	(1) is proved in \cite{Chen20}.
	(2) and (3) follow from Proposition \ref{prop:conjtoconj}, since under the conditions in either (2) or (3), Conjecture \ref{conj: (n,I)complforfiniterat} is just Theorem \ref{thm: (n,I)complforfiniterat}, which is already verified.
\end{proof}

\end{document}